\nonstopmode
\documentclass{amsart}
\usepackage{calc}
\usepackage{amsmath, amssymb, amsthm, amscd}
\usepackage{url}
\usepackage{enumerate}
\usepackage[english]{babel}
\usepackage[all]{xy}

\newtheorem*{remark*}{Remark}

\newtheorem*{prop*}{Proposition}
\newtheorem{thm}[subsection]{Theorem}
\newtheorem*{thm*}{Theorem}

\newtheorem*{lem*}{Lemma}

\newtheorem*{cor*}{Corollary}
\newtheorem*{result*}{Result}

\newtheorem*{ass*}{Assumption}

\def\o{\circ}
\def\X{\mathfrak X}

\def\de{\delta}

\def\ph{\varphi}

\def\Ga{\Gamma}
\def\De{\Delta}

\def\Ph{\Phi}
\def\Ps{\Psi}

\def\i{^{-1}}
\def\x{\times}
\def\p{\partial}
\let\on=\operatorname
\def\L{\mathcal L}
\def\grad{\on{grad}}%
\def\AMSonly#1{}

\def\Id{\on{Id}}
\def\R{\mathbb R}
\def\Tr{\on{Tr}}
\def\vol{{\on{vol}}}
\def\Vol{{\on{Vol}}}
\def\Imm{{\on{Imm}}}

\def\Diff{{\on{Diff}}}
\def\g{\overline{g}}
\def\grad{\on{grad}}

\def\Nor{{\on{Nor}}}

\def\hor{{\on{hor}}}
\def\Length{{\on{Len}}}

\def\undersetbrace#1\to#2{\underbrace{#2}_{#1}}								
\def\oversetbrace#1\to#2{\overbrace{#2}^{#1}}
\def\AMSunderset#1\to#2{\underset{#1}{#2}}
\def\AMSoverset#1\to#2{\overset{#1}{#2}}

\def\norm#1{\left\|{#1}\right\|}
\def\adj#1{\on{Adj}({#1})}

\newdir{ >}{{}*!/-10pt/@{>}}
\newdir^{ (}{{}*!/-5pt/@^{(}}
\newdir_{ (}{{}*!/-5pt/@_{(}}
\def\cgaps#1{}
\def\Cgaps#1{}
\allowdisplaybreaks

\title[Weighted Sobolev Metrics and Almost Local Metrics]
      {Sobolev Metrics on Shape Space, II: \\ Weighted Sobolev Metrics and Almost Local Metrics}
\author{Martin Bauer, Philipp Harms, and Peter~W.~Michor}

\address{
Martin Bauer:
Fakult\"at f\"ur Mathematik, Universit\"at Wien, 
Nordbergstrasse 15, A-1090 Wien, Austria.}
\email{bauer.martin@univie.ac.at}
\address{
Philipp Harms:
EdLabs, Harvard University,
44 Brattle Street, Cambridge, MA 02138, USA.}
\email{pharms@edlabs.harvard.edu}
\address{
Peter W. Michor:
Fakult\"at f\"ur Mathematik, Universit\"at Wien,
Nordbergstrasse 15, A-1090 Wien, Austria.}
\email{peter.michor@univie.ac.at}


\thanks{All authors were supported by
FWF Project 21030. MB was supported by FWF Project P~24625.}

\begin{document}

\maketitle

\begin{abstract}
In continuation of \cite{Michor119} we discuss metrics
of the form
$$
G^P_f(h,k)=\int_M \sum_{i=0}^p\Ph_i\big(\Vol(f)\big)\ \g\big((P_i)_fh,k\big) \vol(f^*\g)
$$
on the space of immersions $\Imm(M,N)$
and on shape space $B_i(M,N)=\Imm(M,N)/\on{Diff}(M)$.
Here $(N,\g)$ is a complete Riemannian manifold, $M$ is a compact manifold,
$f:M\to N$ is an immersion, $h$ and $k$ are tangent vectors to $f$ in the space of immersions,
$f^*\g$ is the induced Riemannian metric on $M$,
$\vol(f^*\g)$ is the induced
volume density on $M$, $\Vol(f)=\int_M\vol(f^*\g)$, $\Ph_i$ are positive real-valued functions, and
$(P_i)_f$ are operators like some power of the Laplacian $\De^{f^*\g}$.
We derive the geodesic equations for these metrics and show that they are sometimes well-posed with the geodesic
exponential mapping a local diffeomorphism.
The new aspect here are the weights $\Ph_i(\Vol(f))$ which we use to construct scale invariant
metrics and order 0 metrics with positive geodesic distance.
We treat several concrete special cases in detail.
\end{abstract}

\section{Introduction}

This paper is a continuation of the article \cite{Michor119} which discussed general Sobolev
metrics on shape space: The special examples of metrics on spaces of shapes in $\mathbb R^n$
studied there are invariant under the motion group, but
not under scalings. The geodesic equation is well posed.

In contrast, the article \cite{Michor118} discusses weighted $L^2$-metrics:
They have easy horizontal bundles, and they can be made scale invariant by a judicious choice of
the weight function. But we do not know that the geodesic equation is well posed.

In this article we discuss in detail several interesting cases of weighted Sobolev metrics
on shape space which are more explicit special cases of the general setup in \cite[section 6]{Michor119}.
In particular we are interested in scale invariant versions. For shapes of planar curves this was
discussed in \cite{Michor107}; detailed studies of concrete examples are in \cite{Michor111} and 
\cite{Michor128}.
Sobolev type metrics on the manifold of all Riemannian metrics were treated in \cite{Michor123}.

\subsection{The shape spaces in this work}
Let $(N,\g)$ be a complete Riemannian manifold, and let $M$ be a compact manifold (the template
shape).
A \emph{shape} is a smoothly embedded surface in $N$ which is diffeomorphic to $M$.
The space of these shapes will be  denoted $B_e=B_e(M,N)$ and viewed as the
quotient (see \cite{Michor102} for more details)
$$ B_e(M,N) = \on{Emb}(M,N)/\on{Diff}(M)$$
of the open subset $\on{Emb}(M,N)\subset C^\infty(M,N)$
of smooth embeddings of $M$ in $N$,
modulo the group of smooth diffeomorphisms of $M$.
It is natural to consider all possible \emph{immersions} as well
as embeddings, and thus introduce the larger
space $B_i=B_i(M,N)$ as the quotient of the space of smooth immersions by the
group of diffeomorphisms of $M$ (which is, however, no longer a manifold,
but an orbifold with finite isotropy groups, see \cite{Michor102}).
\begin{equation*}
\xymatrix{
\on{Emb}(M,N) \ar@{^{(}->}[d] \ar@{->>}[r] &
\on{Emb}(M,N)/\on{Diff}(M) \ar@{^{(}->}[d] \ar@{=}[r]  &
B_e(M,N) \ar@{^{(}->}[d] \\
\Imm(M,N) \ar@{->>}[r] &
\Imm(M,N)/\on{Diff}(M)  \ar@{=}[r]  & B_i(M,N)
}
\end{equation*}
More generally, a shape will be an element of the Cauchy completion (i.e., the metric completion
for the geodesic distance) of $B_i(M,N)$ with
respect to a suitably chosen Riemannian metric. This will allow for corners.
In practice, discretization for numerical algorithms will hide the need to go to the Cauchy
completion.

Since $B_i(M,N)$, as a quotient space or moduli space, is notoriously difficult to work on,
all our computations will be done on the smooth infinite dimensional manifold $\Imm(M,N)$.
In particular we shall use tangent vectors $h$ to $\Imm(M,N)$ with foot point $f\in\Imm(M,N)$ which
are smooth vector fields $h:M\to TN$ along $f$. The tangent bundle $T\Imm(M,N)$ is the smooth
open submanifold of all $h\in C^\infty(M,TN)$ such that $\pi_N\o h$ is in the open subset
$\Imm(M,N)\subset C^\infty(M,N)$.

Any tangent vector $h \in T_f\Imm$ can be decomposed uniquely as $h=Tf.h^\top + h^\bot$, such that
$h^\bot(x)$ is $\g$-orthogonal to $T_xf.T_xM$ at every point $x \in M$.
We call $h^\top$ and $h^\bot$ the tangential and normal parts of $h$, respectively.

\subsection{Sobolev metrics on shape space}\label{intr:sobolev}
As defined in \cite[section 6]{Michor119}, a Sobolev metric
is a $\Diff(M)$-invariant weak Riemannian metric on $\Imm(M,N)$ of the form
$$
G^P_f(h,k)=\int_M \g(P_fh,k) \vol(g) \quad \text{for} \quad f \in \Imm, \quad h,k \in T_f\Imm.
$$
where $g=f^*\g$ is the induced Riemannian metric on $M$, $\vol(g)=\vol(f^*\g)$ is the induced
volume density on $M$, and where
$P$ is a smooth section of the bundle $L(T\Imm;T\Imm)$ over $\Imm$ which is $\Diff(M)$-invariant
such that at every $f \in \Imm$ the operator
$$P_f:T_f\Imm \to T_f\Imm$$
is an elliptic pseudo differential operator of order $2p$ that is symmetric and positive with respect to
the $H^0$-metric on $\Imm$,
$$H^0_f(h,k) = \int_M \g(h,k)\vol(g).$$
(1) If moreover
the smooth \emph{adjoint}
\begin{equation*}
\adj{\nabla P} \in \Ga\big(L^2(T\Imm;T\Imm)\big)
\end{equation*}
of $\nabla P$ exists in the following sense:
\begin{equation*}
\int_M \g\big((\nabla_m P)h,k\big) \vol(g)=\int_M \g\big(m,\adj{\nabla P}(h,k)\big) \vol(g),
\end{equation*}
then the geodesic equation for the Sobolev metric exists, see \cite[6.2--6.5]{Michor119}.

\noindent (2)
In \cite[theorem 6.6]{Michor119} we have proved that the geodesic equation for a Sobolev metric is
well posed in the sense that we have local existence and uniqueness for geodesics in the
$C^\infty$-sense in such a way, that the Riemannian exponential mapping is a local diffeomorphism,
if $P$ satisfies the following further hypothesis:
\newline
$P,\nabla P$ and $\adj{\nabla P}^\bot$ are smooth sections of the bundles
\begin{equation*}\xymatrix{
L(T\Imm;T\Imm) \ar[d] & L^2(T\Imm;T\Imm) \ar[d] & L^2(T\Imm;T\Imm) \ar[d] \\
\Imm & \Imm & \Imm,
}\end{equation*}
respectively.
Viewed locally in trivializations of these bundles,
$$P_f h, \qquad (\nabla P)_f (h,k), \qquad \big(\adj{\nabla P}_f(h,k)\big)^\bot$$
are pseudo-differential operators of order $2p$ in $h,k$ separately.
As mappings in the foot point $f$ they are non-linear, and it is assumed that they are a
composition of operators of the following type:
\newline
\indent\textrm{(a)} Local operators of order $l\le 2p$, i.e., nonlinear differential operators
$$A(f)(x)=A\big(x,\hat \nabla^{l}f(x),\hat \nabla^{l-1}f(x),\dots,\hat \nabla f(x), f(x)\big),$$
\indent \phantom{(a)} where $\hat \nabla$ is a fixed covariant derivative independent of $f$, and
\newline
\indent\textrm{(b)}  Linear pseudo-differential operators of degrees $l_i$,
\newline
such that the total (top) order of the composition is $\le 2p$.

For numerical experiments concerning Sobolev metrics of order one see the articles \cite{Michor119,MBMB2011}.

\subsection{Almost local metrics}
In the paper \cite{Michor118} we studied metrics of the following type on the shape space
$B_i(M,\mathbb R^n)$ (where $\dim(M)=n-1$):
$$
G^\Ph_f(h,k) = \int_M \Ph\big(\Vol(f),\Tr(L^f)\big)\ \g(h,k)\,\vol(f^*\g),\quad f\in\Imm,\ h,k\in T_f\Imm,
$$
where now $\g$ is the Euclidean metric on $\mathbb R^n$, $\vol(f^*\g)$ is the volume density on $M$
induced by the the pullback metric $g=f^*\g$ on $M$, $\Vol(f)=\int_M\vol(f^*\g)$ is the total
$(n-1)$-volume of $f(M)$, $L^f$ is the Weingarten mapping of the immersed hypersurface, $\Tr(L^f)$
is the mean curvature, and where $\Ph:\mathbb R^2\to \mathbb R_{>0}$ is a smooth function.
Note that $\Tr(L^f)$ is a nonlinear differential operator of $f$ and thus local. By calling such a
metric almost local we meant that the only non-local expression in the metric is the total volume
$\Vol(f)$. The investigation in \cite{Michor118} was aimed at curvature computations and numerical
examples. Doing numerics for almost local metrics is facilitated by the fact that the horizontal bundle for
almost local metric is very simple. In the paper \cite{Michor120} this investigation was extended
to metrics of the form
$$ G^\Ph_f(h,k) = \int_{M} \Ph\big(\Tr(L^f),\det(L^f)\big) \ \g(h, k) \on{vol}(f^*\bar{g})$$
where now $\Ph$ is a positive function of the mean curvature and the Gauss curvature.

\subsection{The metrics in this paper: weighted Sobolev metrics}
Here we combine the settings of the papers \cite{Michor119}, \cite{Michor118}, and
\cite{Michor120} in the following way: The ambient space $(N,\g)$ is again a complete Riemannian
manifold, but not necessarily $\R^n$. The metrics are of the form
\begin{equation}\label{1}
G_f(h,k) = \int_{M} \sum_{i=0}^p\Ph_i\big(\Vol(f)\big) \ \g((P_i)_f h, k) \on{vol}(f^*\bar{g}),
\end{equation}
where the
$P_i$ are smooth sections of the bundle $L(T\Imm;T\Imm)$ over $\Imm$ which are $\Diff(M)$-invariant
such that at every $f \in \Imm$ each operator
$$(P_i)_f:T_f\Imm \to T_f\Imm$$
is a pseudo differential operator that is symmetric and positive semidefinite  with respect to
the $H^0$-metric on $\Imm$, where the $\Ph_i$ are positive smooth functions, and where
$\sum_i \Ph_i(\Vol(f)).(P_i)_f$ is elliptic and positive self adjoint.
The operator governing the metric is thus $P_f:=\sum_i\Ph_i(\Vol(f)).(P_i)_f$.
We may again call the corresponding metric $G^P$ an almost local metric if all $(P_i)_f(h)$ are
local operators in $h$ and in $f$, i.e., differential operators according to the linear
(\cite{Peetre59}, \cite{Peetre60}) and non-linear (\cite{Slovak88}) Peetre theorem, and if the only non-local ingredients
are functions of $\Vol(f)$.

In section \ref{sec:weightedsobolev} we discuss general metrics of the form \eqref{1}, deduce the
geodesic equation (\ref{ge_eq_imm}, if the adjoints of $\nabla P_i$ all exist) and remark that it
is well-posed (under the conditions listed in \ref{intr:sobolev}), where we allow
the assumptions (\ref{intr:sobolev}.2) in a more general form:
The linear operator $\vol(f^*\g)\mapsto \int_M\vol(f^*\g)=\Vol(f)$ is not a pseudo-differential
operator in the strict sense since $\de(\xi)$ is never an admissible symbol. But the operator
is very easy to handle and thus  \cite[theorem 6.6]{Michor119} continues to hold.
In \ref{conserv} we discuss the conserved
quantities arising from symmetry groups. Finally we discuss the horizontal bundle and the
horizontal geodesic equation which corresponds to the geodesic equation on shape space $B_i(M,N)$.

Theorem \ref{sec:geod-dist} discusses conditions which imply that geodesic distance on shape space
is positive. We recall the conditions for positivity  given in \cite{Michor119} and add two
conditions for operators of order zero.
This is a nontrivial question, since geodesic distance on shape space vanishes for the
$L^2$-metric (where $P=\on{Id}$) as shown in \cite{Michor98} and \cite{Michor102}, and also on
diffeomorphism groups (see \cite{Michor102}) and on the Virasoro group (see \cite{Michor122}).
We even have vanishing geodesic distance of diffeomorphism groups for Sobolev metrics of order
$0\le p <\frac12$, see \cite{Michor124}.

Section \ref{sec:almostlocal} is devoted to almost local metrics of order 0 in $h$, where the operator
is given by $P_f= \sum_i\Ph_i(\Vol(f))(P_i)_f = \sum_i\Ph_i(\Vol(f)).\Ps_i(f).\on{Id}$
where $\Ps_i(f)$ is a smooth positive function depending smoothly and equivariantly on $f$;
for example, $\Ps(f)= \|\Tr^g(S)\|_g^2$ can be the $g=f^*\g$-norm squared of the vector valued mean
curvature. For these metrics the horizontal bundle at $f$ is just the bundle of vector fields
along $f$ which are $\g$-orthogonal to $f$. In particular, we discuss the case of conformal metrics
where $P_f=\Ph(\Vol(f)).\on{Id}$, and of curvature weighted metrics where
$P=(1+A\|\Tr^g(S)\|^2_g)\on{Id}$.
We derive the geodesic equations for all these metrics. Note that the general well-posedness
theorem \cite[theorem 6.6]{Michor119} is not applicable for these metrics since the order in $h$ is
 not high enough.

Section \ref{sec:scale_inv} is devoted to metrics corresponding to
operators
$$
P_f = \sum_{i=0}^p \Vol(f)^{\frac{2(i-1)}{m}-1} \g((\De^f)^i h, k) \vol(f^*\g),
$$
where $\De^f$ is the Bochner-Laplace operator associated to $f$ and $\g$, acting on vector fields on
$M$ along $f$, see \cite[3.11 and 5.9]{Michor119}.
The powers here are chosen in such a way that for $N=\mathbb R^n$ the metrics are invariant under rescalings.
We derive the geodesic equations on the space of immersions and on shape space. If $p\ge 1$ then
the geodesic equation is well posed.

The final section \ref{sec:variation} is an appendix which recalls the induced covariant derivative
$\nabla^{\g}$ on any bundle over the space $\Imm(M,N)$ for the convenience of the reader. We
differentiate the vector valued mean curvature $\Tr^g(S^f)$ with respect to the immersion $f$ and 
we recall other variational formulas.

\section{Weighted Sobolev type metrics}\label{sec:weightedsobolev}

We want to study metrics that are induced by operators of the form
$$
P_f:=\sum_{i\in I}\Phi_i(\Vol(f))\ (P_i)_f: T_f\Imm(M,N)\mapsto T_f\Imm(M,N).
$$
\begin{ass*}
\emph{
For each $i\in I$ let $\Phi_i: \R^+\mapsto \R^+$ be a smooth function and $P_i$ be a smooth section of the bundle of linear maps $L(T\Imm;T\Imm)$ over $\Imm$.
In addition we assume that $P_f$ is an elliptic pseudo differential operator that
is symmetric and positive  with respect to
the $H^0$-metric on $\Imm$, i.e.:
$$H^0_f(P_f h,k)=H^0_f(h,P_f k)\text{ and }H^0_f(P_f h,h)\geq 0, \text{ for }h,k\in T_f\Imm(M,N).$$}
\end{ass*}

Note that an elliptic symmetric operator is self-adjoint by \cite[theorem 26.2]{Shubin1987}.
Then the operator $P$ induces a  metric on the set of immersions, namely
$$G^P_f(h,k):=\int_M \g(P_fh,k) \vol(g)=\sum_{i\in I}\Phi_i(\Vol)\int_M \g\big((P_i)_fh,k\big) \vol(g)\;,$$
for $f \in \Imm$, and  $h,k \in T_f\Imm$.
The metric $G^P$ is positive definite since $P$ is assumed to be positive with respect to the
$H^0$-metric.

\subsection{Invariance of $P$ under reparametrizations}\label{so:in}

\begin{ass*}
\emph{ It will be assumed that all $P_i$ are invariant under
the action of the reparametrization group $\Diff(M)$ acting on $\Imm(M,N)$, i.e.
\begin{align*}
P_i=(r^{\ph})^* P_i \qquad \text{for all } \ph \in \Diff(M).
\end{align*}
}
\end{ass*}

For any $f \in \Imm$ and $\ph \in \Diff(M)$ this means
$$(P_i)_f = (T_fr^{\ph})\i \o (P_i)_{f \o \ph} \o T_fr^{\ph}.$$
Applied to $h \in T_f\Imm$ this means
$$(P_i)_f(h) \o \ph = (P_i)_{f \o \ph}(h \o \ph).$$
Note that $\Vol$ is invariant under the action of the diffeomorphism group and
therefore all $\Phi_i(\Vol)$ are invariant as well.
This together with the invariance of $P_i$ implies that the induced metric $G^P$ is invariant
under the action of $\Diff(M)$.  Therefore $G^P$ induces a
unique metric on $B_i$ as explained in \cite[section 2]{Michor119}.

\subsection{The adjoint of $\nabla P_i$}\label{so:ad}
Following \cite{Michor119} we introduce the adjoint operator of $\nabla P_i$ to express the metric gradient $H$ which is part of
the geodesic equation. $H_f$ arises from the metric $G_f$ by differentiating it with
respect to its foot point $f \in \Imm$.
Since the metric is defined via the operator $P=\sum_i\Phi_i(\Vol).P_i$, one also needs to differentiate $(P_i)_f$ with respect to its
foot point $f$. As for the metric, this is accomplished by the covariant derivate.
For $P_i \in \Ga\big(L(T\Imm;T\Imm)\big)$ and $m \in T\Imm$ one has
$$\nabla_m P_i \in \Ga\big(L(T\Imm;T\Imm)\big), \qquad \nabla P_i \in \Ga\big(L(T^2\Imm;T\Imm)\big).$$
See~\cite[section 2]{Michor119} for more details.

\begin{ass*}
\emph{
For all $P_i$ there exists a smooth adjoint
$$\adj{\nabla P_i} \in \Ga\big(L^2(T\Imm;T\Imm)\big)$$
of $\nabla P_i$ in the following sense:
\begin{equation*}
\int_M \g\big((\nabla_m P_i)h,k\big) \vol(g)=\int_M \g\big(m,\adj{\nabla P_i}(h,k)\big) \vol(g).
\end{equation*}}
\end{ass*}

The existence of the adjoint needs to be checked in each specific example,
usually by partial integration.

\begin{lem*}[\cite{Michor119}]
If the adjoint of $\nabla P_i$ exists, then its tangential part is determined
by the invariance of $P_i$ with respect to reparametrizations:
\begin{align*}
\adj{\nabla P_i}(h,k)^\top &=\big(\g(\nabla P_ih,k)-\g(\nabla h,P_ik)\big)^\sharp \\
&=\grad^g \g(P_ih,k)-\big(\g(P_ih,\nabla k)+\g(\nabla h,P_ik)\big)^\sharp
\end{align*}
for $f \in \Imm, h,k \in T_f\Imm$.
\end{lem*}

\subsection{The geodesic equation on the manifold of immersions}\label{ge_eq_imm}
As explained in the introduction the geodesic equation can be expressed
in terms of the metric gradients $H$ and $K$ (see \cite[section~4.4]{Michor119} for more details on this topic).
These gradients and thus also the geodesic equation will be computed in this section.

\begin{thm*}
If the adjoint $\adj{\nabla P_i}$ exists for all $i\in I$, then also $H$ and $K$ exist and the geodesic equation is given by:
\begin{align*}
&\nabla_{\p_t} f_t= \frac12 H_f(f_t,f_t) - K_f(f_t,f_t)\\ &=
\frac12 P\i\bigg(\sum_i\Phi_i(\Vol)\Big(\adj{\nabla P_i}(f_t,f_t)^\bot
-
\frac{\Phi_i'(\Vol)}{\Phi_i(\Vol)}\int_M \g(P_if_t,f_t)\vol(g)\Tr^g(S)\\&\qquad\qquad-2Tf.\g(P_if_t,\nabla f_t)^\sharp -\g(P_if_t,f_t).\Tr^g(S)\Big)\bigg)
\\&\quad-
P\i\Big(\sum_i \Phi_i(\Vol)(\nabla_{f_t} P_i)f_t-\sum_i\int_M g(f_t,\Tr^g(S))\vol(g))\Phi_i'(\Vol)P_if_t\\&\qquad\qquad +\Tr^g\big(\g(\nabla f_t,Tf)\big).Pf_t\Big)\;.
\end{align*}
\end{thm*}
\begin{remark*}
Note that we can apply   \cite[theorem~6.6]{Michor119}. Therefore the geodesic equation is
well-posed assuming that the operator $P$ satisfies  all conditions listed in
\cite[section 6.6]{Michor119}, which are reviewed in {\rm (\ref{intr:sobolev}.2)}.
\end{remark*}

\begin{remark*}
As shown in \cite{Michor119} the geodesic equation can be equivalently written in terms of the momentum $p=G(f_t,.)$ as:
\begin{align*}
p &= G(f_t, \cdot) \\
\nabla_{\p_t} p &= \frac12 G_f\big( H(f_t,f_t),\cdot\big),
\end{align*}
For the metric $G^P$ the momentum takes the form
$$p=Pf_t\otimes\vol(g)=\Phi(\Vol)Pf_t\otimes\vol(g): \R \to T^*\Imm$$
since all other parts of the metric (namely the integral and $\g$)
are constant and can be neglected (cf. \cite[section 6.5]{Michor119}).
Then the geodesic equation reads as:
\begin{align*}
p&=\sum_i\Phi_i(\Vol)P_if_t\otimes\vol(g)
\\
\nabla_{\p_t}p &= \frac12\bigg(\sum_i\Phi_i(\Vol)\Big(\adj{\nabla P_i}(f_t,f_t)^\bot
\\&\qquad\qquad\qquad-
\frac{\Phi_i'(\Vol)}{\Phi_i(\Vol)}\int_M \g(P_if_t,f_t)\vol(g)\Tr^g(S)\\&\qquad\qquad\qquad-2Tf.\g(P_if_t,\nabla f_t)^\sharp -\g(P_if_t,f_t).\Tr^g(S)\Big)\bigg)\otimes\vol(g)
\end{align*}
\end{remark*}

\begin{proof}[Proof of Theorem \ref{ge_eq_imm}]
The governing equations for the metric gradients $H,K \in \Ga\big(L^2(T\Imm;T\Imm)\big)$ are
$$(\nabla_m G)(h,k)=G\big(K(h,m),k\big)=G\big(m,H(h,k)\big)\;,$$
where $h,k,m$ are vector fields on $\Imm$ (cf. \cite[section 4.3]{Michor119}).
We calculate:
\begin{equation}\label{Kgradient}
\begin{aligned}
&(\nabla_m G^P)(h,k)=
D_{(f,m)} \int_M \g(Ph,k) \vol(g) \\&\qquad\qquad
- \int_M \g\big(P(\nabla_m h),k\big) \vol(g)
- \int_M \g(Ph,\nabla_m k) \vol(g)
\\&\qquad=
\int_M D_{(f,m)}\g(Ph,k) \vol(g)
+ \int_M \g(Ph,k) D_{(f,m)}\vol(g)\\&\qquad\qquad
- \int_M \g\big(P(\nabla_m h),k\big) \vol(g)
- \int_M \g(Ph,\nabla_m k) \vol(g)
\\&\qquad=
\int_M \g\big(\nabla_m(Ph),k\big) \vol(g)
+ \int_M \g(Ph,\nabla_m k) \vol(g)\\&\qquad\qquad
+ \int_M \g(Ph,k) D_{(f,m)}\vol(g)\\&\qquad\qquad
- \int_M \g\big(P(\nabla_m h),k\big) \vol(g)
- \int_M \g(Ph,\nabla_m k) \vol(g)
\\&\qquad=
\int_M \g\big((\nabla_m P)h,k\big) \vol(g)
+ \int_M \g(Ph,k) D_{(f,m)}\vol(g)\;.
\end{aligned}
\end{equation}
Since $P$ is an elliptic, self-adjoint, and positive operator, it is invertible on the space of smooth
sections. (See the beginning of the proof of \cite[theorem 6.6]{Michor119} for a detailed argument.)
Plugging in the variational formulas from \ref{ap:var} for the volume form and the volume,
one immediately gets the $K$-gradient.
\begin{multline*}
K_f(h,m)=P\i\Big(\sum_i \Phi_i(\Vol)(\nabla_m P_i)h-\sum_i\int_M g(m,\Tr^g(S))\vol(g))\Phi_i'(\Vol)P_ih\\ +\Tr^g\big(\g(\nabla m,Tf)\big).Ph\Big)\;.
\end{multline*}
This formula is equal to the formula presented in \cite[section 6.3]{Michor119}.

To calculate the $H$-gradient, one rewrites equation~\eqref{Kgradient}
using the definition of the adjoint of $P$ (c.f. \cite[section 6.3]{Michor119}):
\begin{align*}
(\nabla_m G^P)(h,k)&=\int_M \g\big(\sum_i\nabla_m (\Phi_i(\Vol)P_i)h,k\big) \vol(g)
+ \int_M \g(Ph,k) D_{(f,m)}\vol(g)\\&=
\sum_i\Phi_i(\Vol)\int_M \g\big(m,\adj{\nabla P_i}(h,k)\big) \vol(g)\\&\qquad
+\sum_i\Phi_i'(\Vol)D_{(f,m)}\Vol \int_M \g(P_ih,k) \vol(g)\\&\qquad
+ \sum_i\Phi_i(\Vol)\int_M \g(P_ih,k) D_{(f,m)}\vol(g).
\end{align*}
Now the second and third summand are treated further using  the variational formulas for the volume
density and volume from \ref{ap:var}. For the second term this yields:
\begin{multline*}
\sum_i\Phi_i'(\Vol)D_{(f,m)}\Vol \int_M \g(P_ih,k)\vol(g)=\\=
-\sum_i\Phi_i'(\Vol)\int_M g(m,\Tr^g(S))\vol(g)) \int_M \g(P_ih,k) \vol(g).
\end{multline*}
For the third term we calculate:
\begin{align*}
&\sum_i\Phi_i(\Vol)\int_M \g(P_ih,k) D_{(f,m)}\vol(g)=\\&\qquad=
\sum_i\Phi_i(\Vol)\int_M \g(P_ih,k) \Tr^g\big(\g(\nabla m,Tf)\big) \vol(g) \\
&\qquad=
\sum_i\Phi_i(\Vol)\int_M \g(P_ih,k) \Tr^g\big(\nabla\g(m,Tf)-\g(m,\nabla Tf)\big) \vol(g) \\
&\qquad=
\sum_i\Phi_i(\Vol)\int_M \g(P_ih,k) \Big(-\nabla^*\g(m,Tf)-\g\big(m,\Tr^g(S)\big)\Big) \vol(g) \\
&\qquad=
-\sum_i\Phi_i(\Vol)\int_M g^0_1\big(\nabla\g(P_ih,k),\g(m,Tf)\big)\vol(g)\\
&\qquad\quad\,-\sum_i\Phi_i(\Vol)\int_M \g(P_ih,k) \g\big(m,\Tr^g(S)\big) \vol(g) \\
&\qquad=
\sum_i\Phi_i(\Vol)\int_M \g\big(m,-Tf.\grad^g\g(P_ih,k)-\g(P_ih,k) \Tr^g(S)\big) \vol(g)
\end{align*}
Collecting terms one gets that
\begin{align*}
&G_f^P(H_f(h,k),m)=(\nabla_m G^P)(h,k)\\&\quad=
\sum_i\bigg(\Phi_i(\Vol)\int_M \g\Big(m,\adj{\nabla P_i}(h,k)-
\frac{\Phi_i'(\Vol)}{\Phi_i(\Vol)}\int_M \g(P_ih,k)\vol(g)\Tr^g(S)\\&\quad\qquad\qquad\qquad\qquad\qquad\qquad\quad
-Tf.\grad^g\g(P_ih,k)-\g(P_ih,k) \Tr^g(S)\Big) \vol(g)\bigg)\;.
\end{align*}
Thus the $H$-gradient is given by
\begin{align*}
&H_f(h,k)=P\i\bigg(\sum_i\Phi_i(\Vol)\Big(\adj{\nabla P_i}(h,k)\\&
-
\frac{\Phi_i'(\Vol)}{\Phi_i(\Vol)}\int_M \g(P_ih,k)\vol(g)\Tr^g(S)
-Tf.\grad^g\g(P_ih,k) -\g(P_ih,k).\Tr^g(S)\Big)\bigg)\;.
\end{align*}
The highest order term $\grad^g\g(P_ih,k)$ cancels out when taking into
account the formula for the tangential part of the adjoint from section~\ref{so:ad}:
\begin{align*}
H_f(h,k)
&=P\i\bigg(\sum_i\Phi_i(\Vol)\Big(\adj{\nabla P_i}(h,k)^\bot
\\&\qquad
-\frac{\Phi_i'(\Vol)}{\Phi_i(\Vol)}\int_M \g(P_ih,k)\vol(g)\Tr^g(S)
\\&\qquad
-Tf.\big(\g(P_ih,\nabla k)+\g(\nabla h,P_ik)\big)^\sharp -\g(P_ih,k).\Tr^g(S)\Big)\bigg)
\;.
\qedhere
\end{align*}
\end{proof}

\subsection{Conserved Quantities}\label{conserv}
The metric $G$ is invariant under the action of the reparametrization group on $M$.
According to \cite{Michor118} the momentum mapping for this group
action is constant along any geodesic in $\on{Imm}(M,N)$.:
$$\boxed
{\begin{aligned}
\forall X\in\X(M): G( Tf.X,f_t ) && \text{rep. mom.}\\
\text{or }G(f_t^\top,\cdot)   \in\Ga(T^*M\otimes_M\on{vol}(M))
	&& \text{rep. mom.}\\
\end{aligned}}$$
If $(N,\g)$ has a non-trivial isometry group, this gives rise to further conserved quantities
with values in the Lie-algebra of Killing vector fields.
For example, take a flat ambient space $N=\R^n$. Then the almost local metrics are in addition invariant under the action of the
Euclidean motion group $\mathbb R^n \rtimes \on{SO}(n)$. This yields the following conserved quantities:
$$\boxed
{\begin{aligned}
\int_M \sum_{i=1}^p \Phi_i\big(\on{Vol(f)}\big) P_if_t \on{vol}(g) && \text{lin. mom.}\\
\forall X\in \mathfrak{so}(n): \int_M \sum_{i=1}^p \Phi_i\big(\on{Vol(f)}\big)  \g( P_if_t,Xf )
\on{vol}(g) && \text{ang. mom.} \\
\text{or }\int_M \sum_{i=1}^p \Phi_i\big(\on{Vol(f)}\big) (f\wedge P_i f_t)
\on{vol}(g) \in {\textstyle\bigwedge^2}\mathbb R^n\cong \mathfrak s\mathfrak o(n)^* && \text{ang. mom.}
\end{aligned}}$$
\subsection{The horizontal bundle and the geodesic equation on $B_i$}\label{ge_eq_bi}
According to \cite{Michor119} the geodesic equation on $B_i$ is equivalent to the horizontal geodesic equation on $\Imm$ assuming that in every equivalence class of curves there exists a horizontal curve. Therefore we need to study the horizontal bundle of the metric
$G^P$. Multiplication with an everywhere positive function does not change the notion of horizontality.
Therefore we can repeat the analysis of \cite[sections 6.8-6.9]{Michor119} to obtain:
\begin{lem*}\label{hor_bun}
The decomposition of $h \in T_f\Imm$ into its vertical and horizontal components
is given by
\begin{align*}
h^{\text{ver}} &= (P_f^\top )\i\big((P_fh)^\top\big),
\\
h^{\text{hor}} &= h - Tf.h^{\text{ver}} =  h - Tf.(P_f^\top )\i\big((P_fh)^\top\big).
\end{align*}
Furthermore for any smooth path $f$ in $\Imm(M,N)$ there exists a
smooth path $\ph$ in $\on{Diff}(M)$ with $\ph(0,\;.\;)=\on{Id}_M$
depending smoothly on $f$ such that
the path $\tilde f$ given by $\tilde f(t,x)=f(t,\ph(t,x))$ is horizontal:
$$G^P(\p_t\tilde f,T\tilde f.TM \big) =0.$$
Thus any path in shape space can be lifted to a horizontal path of immersions.
\end{lem*}

See \cite[sections 6.8-6.9]{Michor119} for a proof of the lemma.
Note that the above lemma  establishes a one-to-one correspondence between horizontal
curves in $\Imm$ and curves in shape space. Thus the geodesic equation on shape space is given by:
\begin{thm*}\label{ge_eq_ge_sh}
The geodesic equation for the $G^P$-metric on shape space is equivalent to the set of equations
for a path of immersions $f$:
\begin{align*}
p &= P f_t \otimes \vol(g), \qquad Pf_t = (Pf_t)^\bot, \\
(\nabla_{\p_t}p)^\hor&= \frac12\bigg(\sum_i\Phi_i(\Vol)\Big(\adj{\nabla P_i}(f_t,f_t)^\bot
\\&\qquad\qquad\qquad-
\frac{\Phi_i'(\Vol)}{\Phi_i(\Vol)}\int_M \g(P_if_t,f_t)\vol(g)\Tr^g(S)\\&\qquad\qquad\qquad\qquad\qquad\qquad\qquad\qquad
 -\g(P_if_t,f_t).\Tr^g(S)\Big)\bigg)\otimes\vol(g)
\end{align*}
\end{thm*}
The proof is a direct consequence of theorem~\ref{ge_eq_imm} and of lemma~\ref{hor_bun}.

\begin{thm}[Geodesic distance]\label{sec:geod-dist}
The  metric $G^P$ induces non-vanishing geodesic distance on $B_e$ if at least one of the following conditions is satisfied
for $C_1,C_2,C_3,A>0$ and for all $h \in T\Imm$:
\begin{enumerate}
\item $\norm{h}_{G^P} \geq C_1 \norm{h}_{H^1} =
C_1 \sqrt{\int_M \g\big( (1+\Delta) h,h \big) \vol(g),}$
\item $\norm{h}_{G^P} \geq C_2 \Vol \norm{h}_{H^0} =
C_2 \sqrt{\Vol(f)\int_M \g\big(h,h \big) \vol(g)},$
\item $\norm{h}_{G^P} \geq C_3  \norm{h}_{G^A} =
C_3 \sqrt{\int_M(1+A\|\Tr^g(S)\|^2) \g\big(h,h \big) \vol(g)}$
\end{enumerate}
\end{thm}
\begin{proof}
For the proof using condition (1) we refer to \cite[section 7]{Michor119}. If condition (2) holds
we calculate for a horizontal path $f$:
\begin{align*}
\Length_{G^P}^{\Imm}(f)&\geq C_2\int_0^1 \sqrt{\Vol(f)\int_M \g\big(f_t,f_t \big) \vol(g)}dt
\geq C_2\int_0^1 \int_M 1.\sqrt{\g\big(f_t,f_t \big)} \vol(g)dt\\
&=C_2\int_{[0,1]\times M}\vol(f(.,.)^*\g)=C_2 \ (\text{area swept out by $f$}) \;.
\end{align*}
Since area swept out is a positive distance on $B_e$ this shows the statement for condition (2).
For condition (3) we calculate:
\begin{align*}
\Length_{G^P}^{\Imm}(f)&\geq C_3\int_0^1 \sqrt{\int_M(1+A\|\Tr^g(S)\|^2) \g\big(f_t,f_t \big) \vol(g)}dt\\&\geq C_3 \int_0^1 \sqrt{\int_M\g\big(f_t,f_t \big) \vol(g)}dt
\geq C_3\int_0^1 \frac{1}{\sqrt{\Vol}}\int_M 1.\sqrt{\g\big(f_t,f_t \big)} \vol(g)dt\\
&\geq C_3\frac{1}{\underset{t}{\max} \sqrt{\Vol(f(t,.))}} \ (\text{area swept out by $f$})
\end{align*}
The statement for condition (3) will be proven by showing that $\sqrt{\Vol(f(t,.))}$ is Lipschitz continuous.
Using the formula from section~\ref{ap:var} for the variation of the volume we get
\begin{align*}
\p_t \on{Vol}(f) &=
-\int_M \g(f_t,\Tr^g(S)) \on{vol}(g) \leq
\left| \int_M \g(f_t,\Tr^g(S)) \on{vol}(g) \right| \\&\leq
\Big(\int_M 1^2 \on{vol}(g)\Big)^{\frac12} \Big(\int_M \g(f_t,\Tr^g(S))^2 \on{vol}(g)\Big)^{\frac12} \\&\leq
\sqrt{\on{Vol}(f)}\Big(\int_M \norm{f_t}^2_{\g}\norm{\Tr^g(S)}^2_{\g} \on{vol}(g)\Big)^{\frac12} \\&\leq
\sqrt{\on{Vol}(f)} \Big(\int_M \frac{\Phi}{C_2} \norm{f_t}^2_{\g} \on{vol}(g)\Big)^{\frac12} \leq
\frac{1}{\sqrt{C_2}}\sqrt{\on{Vol}(f)} \sqrt{G_f(f_t,f_t)}.
\end{align*}
Thus
\begin{align*}
\p_t \sqrt{\on{Vol}(f)}=\frac{\p_t \on{Vol}(f)}{2 \sqrt{\on{Vol}(f)}}\leq
\frac{1}{2\sqrt{C_2}} \sqrt{G_f(f_t,f_t)}.
\end{align*}
By integration we get
\begin{align*}
\sqrt{\on{Vol}(f_1)}-\sqrt{\on{Vol}(f_0)} &=
\int_0^1 \p_t \sqrt{\on{Vol}(f)}dt \\&\leq
\int_0^1 \frac{1}{2\sqrt{C_2}} \sqrt{G_f(f_t,f_t)} =
\frac{1}{2\sqrt{C_2}}\Length^{G}_{\on{Imm}}(f).\qedhere
\end{align*}
\end{proof}

\section{Almost local metrics of order zero}\label{sec:almostlocal}
In this chapter we will study metrics that are induced by operators of order zero, i.e.
$$P=\sum_i \Phi(\Vol(f)).(P_f)_i(h)=\sum_i \Phi(\Vol(f)).\Psi_i(f).h,$$
where $\Psi: \Imm(M,N)\mapsto \R_{>0}$ is any smooth,
positive function depending on the immersion $f$, that is equivariant with respect to the action of the
diffeomorphism group $\Diff(M)$.
Note that $\Psi$ might also be a differential operator acting on the immersion $f$.
So order zero means order zero in $h,k$, but possibly higher order in the foot point $f$.
The most prominent examples of such operators, and also the examples we will consider in this section
are
$$P(f):=\Phi(\Vol)\Psi(\|\Tr^g(S)\|^2).$$
Here $\Tr^g(S)$ denotes the vector valued mean curvature.

For the case of planar curves, metrics of this form have been introduced in \cite{Michor107}. Recently
they have  been generalized to hypersurfaces in $n$-space, see  \cite{Michor118}.
The most general case of surfaces in a possibly curved ambient space has been studied in the PhD thesis of Martin Bauer
\cite{Bauer2010}.

The great advantage of this class of metrics is the particularly simple form of the horizontal bundle, namely:
\begin{lem*}[Horizontal bundle]\label{lem:ho}
For an almost local metric  of order zero, the horizontal bundle at an immersion $f$ equals the set of
sections of the normal bundle along $f$.
\end{lem*}
\begin{proof}
By definition, a tangent vector $h$ to $f \in \on{Imm}(M,N)$ is horizontal if
and only if it is $G$-perpendicular to
the $\on{Diff}(M)$-orbits. This is the case if and only if $\g( h(x), T_x f .X_x ) =
0$ at every point $x \in M$.
\end{proof}

This observation is of particular interest both for numerical and theoretical reasons.
The simple form of the horizontal bundle allows to solve
the boundary value problem for geodesics on shape space numerically by minimizing the horizontal energy
(c.f. \cite{Michor118}).
And it allows for the computation of the sectional curvature, as done in \cite{Michor107}
for the case of planar curves and in \cite{Michor118} for the case of hypersurfaces in $n$-space.
The main drawback of this class of metrics is that it is unclear if the geodesic equation is well posed.
(The operator $P$ does not satisfy the conditions of \cite[theorem 6.6]{Michor119}.)

In the following we want to study two examples of almost local metrics in more detail.

\subsection{Conformal metrics}
The class of conformal metrics  correspond to  metrics $G^P$ with
$$
P=\Phi(\Vol).\Id\;.
$$
For the case of planar curves these metrics have been treated in \cite{YezziMennucci2004},
\cite{YezziMennucci2004a}, \cite{YezziMennucci2005}, \cite{Shah2008}
and for the case of  of hypersurfaces in $n$-space they have been studied in \cite{Michor118}.
In the case of planar curves \cite{Shah2008} provides very interesting estimates on geodesic distance induced by metrics with
$\Ph(\Vol)=\Vol$ and $e^{\Vol}$.

According to theorem~\ref{ge_eq_ge_sh} we can simply read off the geodesic equation both on the manifold of immersions and on shape space.
It is given by:
\begin{thm*}\label{ge_eq_conformal}
The geodesic equation on the manifold of immersions for the class of conformal metrics is given by:
\begin{align*}
p&=\Phi(\Vol)f_t\otimes\vol(g)
\\
\nabla_{\p_t}p &= -\frac12\bigg(
\Phi'(\Vol)\int_M \g(f_t,f_t)\vol(g)\Tr^g(S)\\&\qquad\qquad+2\Phi(\Vol)Tf.\g(f_t,\nabla f_t)^\sharp +\Phi(\Vol)\g(f_t,f_t).\Tr^g(S)\Big)\bigg)\otimes\vol(g).
\end{align*}
If one passes to the quotient space $B_i(M,N)$ the geodesic equation reduces to:
\begin{align*}
p &= \Phi(\Vol)f_t \otimes \vol(g), \qquad f_t = f_t^\bot\in\Nor(f), \\
(\nabla_{\p_t}p)^\hor &= -\frac12\Big(
\Phi'(\Vol)\int_M \g(f_t,f_t)\vol(g)\Tr^g(S)\\&\qquad\qquad\qquad\qquad\qquad\qquad
+\Phi(\Vol).\g(f_t,f_t).\Tr^g(S)\Big) \otimes \vol(g).
\end{align*}
\end{thm*}
\begin{remark*}
Note that the geodesic equation for conformal metrics does not change when passing from a flat to a curved  ambient space.
\end{remark*}

\subsection{Curvature weighted metrics: The $G^A$-metric.}

In our setting the curvature weighted metric $G^A$ corresponds to the operator
$$P=(1+A\|\Tr^g(S)\|^2)\;.$$
Note that $P$ is a differential operator of order two applied to the foot point $f$,
but an operator of order zero in $h$ and $k$.
Metrics weighted by curvature have been studied in  \cite{Michor98}, \cite{Michor107},
\cite{Michor118}, \cite{Michor120}.

According to section~\ref{ge_eq_ge_sh} we need to calculate the adjoint of $P$ to obtain
the  geodesic equation.
\begin{lem*}
For the operator $P=1+A\|\Tr^g(S)\|^2$ the adjoint is given by:
\begin{align*}
&\adj{\nabla P_i}(h,k)=
4A.\Tr\Big(g\i Sg\i\g\big(S,\Tr^g(S)\big)\g(h,k)\Big)\\&\qquad\qquad
-2A\Big(\Delta\Big(\Tr^g(S).\g(h,k)\Big)\Big)^{\bot}
-2A\g(h,k)\Tr^g\Big(R^{\g}\big(Tf,\Tr^g(S)\big)Tf\Big)^\bot\\&\qquad\qquad+ATf.\g(h,k).\on{grad}^g\norm{\Tr^g(S)}^2_{\g}
\end{align*}
\end{lem*}
\begin{proof}
Using the variational formula for the mean curvature (see \ref{ap:var}) we calculate:

\begin{align*}
&\int_M \g\big(m,\adj{\nabla P_i}(h,k)\big) \vol(g)=\int_M \g\big((\nabla_m P_i)h,k\big) \vol(g)\\&\qquad=
A\int_M \big(D_{f,m}\|\Tr^g(S)\|^2\big) \g\big(h,k\big) \vol(g)\\&\qquad
=2A\int_M \g\big(\Tr^g(S),D_{f,m}\Tr^g(S)\big) \g\big(h,k\big) \vol(g)\\
&\qquad=
2A\int_M\g\Big(2\Tr\big(\g(m^{\bot},g\i.S.g\i).S\big),\Tr^g(S)\Big).\g(h,k) \on{vol}(g)\\
&\qquad\qquad
-2A\int_M(\g\otimes g^0_0)\big(\Delta(m^{\bot}),\Tr^g(S)\big).\g(h,k) \on{vol}(g)\\
&\qquad\qquad+2A\int_M \g\Big(\Tr^g(R^{\g}(m^\bot,Tf)Tf),\Tr^g(S)\Big)\g(h,k)\vol(g)\\
&\qquad\qquad
+2A\int_M \frac12 d\norm{\Tr^g(S)}^2_{\g}(m^\top).\g(h,k)\vol(g)\\&\qquad=
4A\int_M\Tr\Big(\g(m^{\bot},g\i.S.g\i).\g\big(S,\Tr^g(S)\big)\Big)\g(h,k)\vol(g)\\&\qquad\qquad-
2A\int_M(\g\otimes g^0_0)\Big(\Delta(m^{\bot}),\Tr^g(S)\Big)\g(h,k)\vol(g)\\
&\qquad\qquad+2A\int_M \Tr^g\Big(\g\big(R^{\g}(m^\bot,Tf)Tf,\Tr^g(S)\big)\Big)\g(h,k)\vol(g)\\
&\qquad\qquad
+A \int_M g\Big(\on{grad}^g\norm{\Tr^g(S)}^2_{\g}, m^\top \Big).\g(h,k)\vol(g)\\&\qquad=
4A\int_M\g\bigg(m^{\bot},\Tr\Big(g\i.S.g\i.\g\big(S,\Tr^g(S)\big).\g(h,k)\Big)\bigg)\vol(g)\\&\qquad\qquad-
2A\int_M(\g\otimes g^0_0)\Big(m^{\bot},\Delta\big(\Tr^g(S).\g(h,k)\big)\Big)\vol(g)\\
&\qquad\qquad-2\int_M A\Tr^g\Big(\g\big(R^{\g}(Tf,\Tr^g(S))Tf,m^\bot\big)\Big)\g(h,k)\vol(g)\\
&\qquad\qquad
+A\int_M  \g\Big(Tf.\on{grad}^g\norm{\Tr^g(S)}^2_{\g}, Tf.m^\top \Big).\g(h,k)\vol(g)\\&\qquad=
4A\int_M\g\bigg(m^{\bot},\Tr\Big(g\i.S.g\i.\g\big(S,\Tr^g(S)\big).\g(h,k)\Big)\bigg)\vol(g)\\&\qquad\qquad-
2A\int_M\g\Big(m^{\bot},\Delta\big(\Tr^g(S).\g(h,k)\big)\Big)\vol(g)\\
&\qquad\qquad-2\int_M \g\Big(\g(h,k)A\Tr^g\big(R^{\g}(Tf,\Tr^g(S))Tf\big),m^\bot\Big)\vol(g)\\
&\qquad\qquad
+A\int_M \g\Big(m,Tf.\g(h,k).\on{grad}^g\norm{\Tr^g(S)}^2_{\g}\Big)\vol(g)\qedhere
\end{align*}
\end{proof}
Using the above lemma and theorem~\ref{ge_eq_ge_sh} yields the geodesic equation for the $G^A$-metrics both on the manifold of immersions and
on shape space:
\begin{thm*}\label{ge_eq_ge_sh_GA}
The geodesic equation for the $G^A$-metric on $\Imm(M,N)$ is given by
\begin{align*}
p&=(1+A\|\Tr^g(S)\|^2) f_t \otimes \vol(g)
\\
\nabla_{\p_t}p &= \frac12\bigg(2A.\Tr\Big(g\i Sg\i\g\big(S,\Tr^g(S)\big)\g(h,k)\Big)
-A\Big(\Delta\Big(\Tr^g(S).\g(h,k)\Big)\Big)^{\bot}\\&\qquad
-A\g(h,k)\Tr^g\Big(R^{\g}\big(Tf,\Tr^g(S)\big)Tf\Big)^\bot\\&\qquad
-(1+A\|\Tr^g(S)\|^2)\g(f_t,f_t).\Tr^g(S)
\\&\qquad-2(1+A\|\Tr^g(S)\|^2)Tf.\g(f_t,\nabla f_t)^\sharp \bigg)\otimes\vol(g)
\end{align*}
The geodesic equation for the $G^A$-metric on shape space is equivalent to the set of equations
for a path of immersions $f$:
\begin{align*}
p &= (1+A\|\Tr^g(S)\|^2) f_t \otimes \vol(g), \qquad f_t = f_t^\bot\in\Nor(f), \\
(\nabla_{\p_t}p)^\hor&=
\bigg(
2A.\Tr\Big(g\i Sg\i\g\big(S,\Tr^g(S)\big)\g(h,k)\Big)
-A\Big(\Delta\Big(\Tr^g(S).\g(h,k)\Big)\Big)^{\bot}\\&\qquad
-A\g(h,k)\Tr^g\Big(R^{\g}\big(Tf,\Tr^g(S)\big)Tf\Big)^\bot
\\&\qquad
 -\frac12(1+A\|\Tr^g(S)\|^2)\;\g(f_t,f_t),\Tr^g(S)\Big)\bigg)\otimes\vol(g)
\end{align*}
\end{thm*}

\section{Scale invariant Sobolev metrics}\label{sec:scale_inv}
In this section we will study scale invariant Sobolev type metrics of order $p$. This class of metrics is induced by an operator of the form
$$
P=\sum_{i=0}^{p} \Vol^{-\frac{2(i-1)}{m}-1}\Delta^i=
\sum_{i=0}^p\Phi_i(\Vol)P_i,
$$
with $$\Phi_i(\Vol)= \Vol^{-\frac{2(i-1)}{m}-1}\quad \text{ and }\quad P_i=\Delta^i.$$
Recall that $m$ denotes the dimension of the base space, i.e. $\dim(M)=m$.
\begin{remark*}
There are other choices of scale invariant metrics. For example, one could use integrals over curvature
terms as weights. A scale invariant metric of order zero that induces positive geodesic distance is induced by the operator:
$$P=\Vol^{-\frac{2}{m}-1}+\Vol^{-1}\|\Tr^g(S)\|^2\;.$$
However, in this paper we decided to focus on
volume weighted Sobolev metrics since they seem to be the most natural examples.
\end{remark*}
In order to calculate the geodesic equation we need to calculate the adjoint operators.
\begin{lem*}
The adjoint of $\nabla P_i$ defined in section~\ref{so:ad} for the operator $P_i=\De^i$ is given by:
\begin{align*}
\adj{\nabla \De^i}(h,k)&=
2\sum_{l=0}^{i-1}\Tr\big(g\i S g\i \g(\nabla\Delta^{i-l-1}h,\nabla\Delta^{l}k ) \big)
\\&\qquad
+\sum_{l=0}^{i-1} \big(\nabla^*\g(\nabla\Delta^{i-l-1}h,\Delta^{l}k) \big) \Tr^g(S)\\
&\qquad+\sum_{l=0}^{i-1}\Tr^g\big(R^{\g}(\Delta^{i-l-1}h,\nabla\Delta^{l}k)Tf \big)\\
&\qquad-\sum_{l=0}^{i-1}\Tr^g\big(R^{\g}(\nabla\Delta^{i-l-1}h,\Delta^{l}k)Tf \big)
\\&\qquad
+Tf.\Big[\grad^g \g(\Delta^i h,k)-\big(\g(\Delta^i h,\nabla k)+\g(\nabla h,\Delta^i k)\big)^\sharp\Big].
\end{align*}
\end{lem*}
\begin{proof}
The proof of this lemma is similar to the calculation of the adjoint for $P=1+A\Delta^p$ in \cite[lemma 8.2]{Michor119}.
\end{proof}
Using this we can calculate the geodesic equation for scale invariant Sobolev type metrics.
\begin{thm*}
The geodesic equation for the weighted Sobolev-type metric of order $p$ on shape space is equivalent to the set of equations
for a path of immersions $f$:
\begin{align*}
p &=Pf_t= \sum_{i=0}^{p} \Vol^{-\frac{2(i-1)}{m}-1}\Delta^i f_t \otimes \vol(g), \qquad Pf_t = f_t^\bot\in\Nor(f), \\
(\nabla_{\p_t}p)^\hor&= \frac12\sum_{i=0}^p \Vol^{-\frac{2(i-1)}{m}-1} \bigg(
2\sum_{l=0}^{i-1}\Tr\big(g\i S g\i \g(\nabla\Delta^{i-l-1}f_t,\nabla\Delta^{l}f_t ) \big)
\\&\qquad\qquad\qquad
+\sum_{l=0}^{i-1} \big(\nabla^*\g(\nabla\Delta^{i-l-1}f_t,\Delta^{l}f_t) \big) \Tr^g(S)\\&\qquad\qquad\qquad
+\sum_{l=0}^{i-1}\Tr^g\big(R^{\g}(\Delta^{i-l-1}f_t,\nabla\Delta^{l}f_t)Tf \big)\\
&\qquad\qquad\qquad-\sum_{l=0}^{i-1}\Tr^g\big(R^{\g}(\nabla\Delta^{i-l-1}f_t,\Delta^{l}f_t)Tf \big)
\\&\qquad\qquad\qquad-
\frac{1}{\Vol}\int_M \g(\Delta^i f_t,f_t)\vol(g)\Tr^g(S)\\&\qquad\qquad\qquad\qquad\qquad\qquad\qquad
 -\g(\Delta^i f_t,f_t).\Tr^g(S)\Big)\bigg)\otimes\vol(g)
\end{align*}
For $p\geq 1$ the operator $P$ satisfies all conditions of \cite[theorem 6.5]{Michor119} and therefore the geodesic equation is well-posed.
\end{thm*}

\section{Appendix: The covariant derivative over immersions and some variational
formulas}\label{sec:variation}

In this appendix we will first recall a concept that has been introduced in \cite{Michor119},
namely a covariant derivative on the manifold of immersions that is induced by the metric $\g$ on the ambient space $N$.
This covariant derivative is used to calculate the metric gradients and to express the geodesic equation
both on the manifold of immersions and on shape space.
In the second part we will present some variational formulas, mainly without proof.
For the missing proofs see for example \cite{Michor118}, \cite{Michor119}, \cite{Michor120},
\cite{Besse2008}, \cite{Bauer2010}, \cite{Harms2010}.

\subsection{Covariant derivative  on immersions}\label{ap:cov}
This section is taken from \cite{Michor119}.
Let $\nabla^{\g}$ be the covariant derivative on $N$ with respect to the metric $\g$.
This covariant derivative induces a
\emph{covariant derivative over immersions} as follows.
Let $Q$ be a smooth manifold. Then one identifies
\begin{align*}
&h \in  C^\infty\big(Q,T\Imm(M,N)\big) && \text{and} && X \in \X(Q)
\intertext{with}
&h^{\wedge} \in C^\infty(Q \x M, TN) && \text{and} && (X,0_M) \in \X(Q \x M).
\end{align*}
Using the covariant derivative
$$\nabla^{\g}_{(X,0_M)} h^{\wedge} \in C^\infty\big(Q \x M, TN),$$
one can define
$$\nabla_X h = \left(\nabla^{\g}_{(X,0_M)} h^{\wedge}\right)^{\vee} \in C^\infty\big(Q,T\Imm(M,N)\big).$$
This covariant derivative is torsion-free (see \cite{Michor119}).
It respects the metric $\g$ but in general does not respect $G$.

It is helpful to point out some special cases of how this construction can be used.
The case $Q=\R$ will be important to formulate the geodesic equation.
The expression that will be of interest in the formulation of the
geodesic equation is $\nabla_{\p_t} f_t$, which is
well-defined when $f:\R \to \Imm$ is a path of immersions and $f_t: \R \to T\Imm$ is its velocity.

Another case of interest is $Q = \Imm$. Let $h, k, m \in \X(\Imm)$. Then the covariant
derivative $\nabla_m h$ is well-defined and tensorial in $m$.
Requiring $\nabla_m$ to respect the grading of the spaces of multilinear maps, to act as a derivation
on products and to commute with compositions of multilinear maps, one obtains a covariant
derivative $\nabla_m$ acting on all mappings into
the natural bundles of multilinear mappings over $\Imm$.
In particular, $\nabla_m P$ and $\nabla_m G$ are well-defined for
\begin{align*}
P \in \Ga\big(L(T\Imm;T\Imm)\big), \quad
G \in \Ga\big(L^2_{\on{sym}}(T\Imm;\R)\big)
\end{align*}
by the usual formulas
\begin{align*}
(\nabla_m P)(h) &= \nabla_m\big(P(h)\big) - P(\nabla_mh), \\
(\nabla_m G)(h,k) &= \nabla_m\big(G(h,k)) - G(\nabla_m h,k) - G(h,\nabla_m k).
\end{align*}

\subsection{Variational formulas}\label{ap:var}
Recall that many operators like
$$g=f^*\g, \quad S=S^f, \quad \on{vol}(g), \quad \nabla=\nabla^g, \quad \Delta=\Delta^g, \quad \ldots$$
implicitly depend on the immersion $f$. In this section we will present their derivative
with respect to $f$, which we call \emph{the first  variation.}
The following lemma contains a compilation of results that
can be found for example in \cite{Harms2010}, \cite{Michor119}.

\begin{result*}
Let $f$ be an immersion and $f_t \in T_f\Imm$ a tangent vector to $f$.
We have:
\begin{enumerate}
\item Let   $F:\on{Imm}(M,N) \to \Gamma(T^r_s M)$ be a smooth mapping, that is  equivariant
  with respect to pullbacks by diffeomorphisms of $M$. Then the tangential variation of $F$ is its Lie-derivative:
  \begin{align*}
  D_{(f,Tf.f_t^\top)} F&=\L_{f_t^\top}\big(F(f)\big).
\end{align*}
\item The differential of the pullback metric $f^*\g$ is given by
  \begin{align*}
  D_{(f,f_t)} g&= 2\on{Sym}\g(\nabla f_t,Tf) = -2 \g(f_t^\bot,S)+2 \on{Sym} \nabla (f_t^\top)^\flat
  \\& = -2 \g(f_t^\bot,S)+ \L_{f_t^\top} g.
  \end{align*}
\item The differential of the inverse of the pullback metric is given by
\begin{align*}
D_{(f,f_t)} g\i = D_{(f,f_t)} (f^*\g)\i =2 \g(f_t^\bot, g\i S  g\i) + \mathcal L_{f_t^\top}(g\i).
\end{align*}
\item\label{variation:volume_form} The differential of the volume density
is given by
\begin{equation*}
D_{(f,f_t)} \on{vol}(g) =
\Tr^g\big(\g(\nabla f_t,Tf)\big) \on{vol}(g)=
\Big(\on{div}^{g}(f_t^{\top})-\g\big(f_t^{\bot},\Tr^g(S)\big)\Big) \on{vol}(g).
\end{equation*}
\item The differential of the total Volume
is given by
$$D_{(f,f_t)} \on{Vol}(f) = D_{(f,f_t)} \int_M \on{vol}(g)
=-\int_M \g\big(f_t^{\bot},\Tr^g(S)\big) \on{vol}(g). $$
\item For $\De \in \Ga\big(L(T\Imm;T\Imm)\big)$, $h \in T_f\Imm$ one has
\begin{align*}
(\nabla_{f_t} \Delta)(h) &=
\on{Tr}\big(g\i.(D_{(f,f_t)}g).g\i \nabla^2 h\big)
-\nabla_{\big(\nabla^*(D_{(f,f_t)} g)+\frac12 d\on{Tr}^g(D_{(f,f_t)}g)\big)^\sharp}h \\&\qquad
+\nabla^*\big(R^{\g}(f_t,Tf)h\big)
-\Tr^g\Big( R^{\g}(f_t,Tf)\nabla h \Big).
\end{align*}
\end{enumerate}
\end{result*}
For the proof of the above  formulas in a similar notation we refer the reader to  \cite{Michor119,Harms2010}.
The following lemma is concerned with the first variation of the vector valued mean curvature:
\begin{lem*}
The differential of the vector valued mean curvature
$$
\on{Imm} \to \Gamma(\Nor(f)),\quad
f \mapsto \Tr^g(S)
$$
is given by
\begin{align*}
D_{(f,f_t)}\Tr^g(S)
&=\Tr\big(2\g(f_t^{\bot},g\i.S.g\i).S\big)-\Delta(f_t^{\bot})\\&\qquad+\Tr^g(R^{\g}(f_t^\bot,Tf)Tf)+\L_{f_t^{\top}}\Tr^g(S).
\end{align*}
\end{lem*}
\begin{proof}
By formula 1 of the above result, the formula for the tangential variation
follows from the equivariance of the vector valued mean curvature  form with respect to pullbacks by
diffeomorphisms. For the variation in normal direction we calculate:
\begin{align*}
\p_t \Tr^g(S)&
=\p_t \Tr(g\i.S)= \Tr\big((\p_tg\i).S\big)+ \Tr\big(g\i.\p_t S\big)
\end{align*}
Using the formula for the variation of the inverse metric it remains to calculate the derivative (in normal direction) of the vector valued shape form.
By definition $S(X,Y)=\nabla_X (Tf.Y)-Tf.\nabla_X Y$. Therefore we have
\begin{align*}
\p_t S^f(X,Y)
&=\nabla_{\p_t}\nabla_X (Tf.Y)-\nabla_{\p_t}Tf.\nabla_X Y\\
&=\nabla_X\nabla_Y Tf.\p_t+R^{\g}(Tf\p_t,TfX)TfY-\nabla_{\nabla_X Y}Tf.\p_t\\
&=\nabla_X\nabla_Y f_t-\nabla_{\nabla_X Y}f_t+R^{\g}(f_t,TfX)TfY\\
&=\nabla^2 f_t+R^{\g}(f_t,Tf)Tf,
\end{align*}
where we used the following two formulas for interchanging covariant derivatives (see \cite{Michor119}):
\begin{equation*}
\nabla_{(\p_t,0_M)} Tf.(0_{\R},Y)-\nabla_{(0_{\R},Y)} Tf.{(\p_t,0_M)} = 0
\end{equation*}
and
\begin{equation*}
\nabla_{(\p_t,0_M)} \nabla_{(0_\R,Y)} h - \nabla_{(0_\R,Y)} \nabla_{(\p_t,0_M)} h
\\= R^{\g} \big(Tf.(\p_t,0_M),Tf.(0_\R,Y)\big) h .
\end{equation*}
Thus we get
\begin{align*}
\p_t \Tr^g(S)&
=\p_t \Tr(g\i.S)= \Tr\big((\p_tg\i).S\big)+ \Tr\big(g\i.\p_t S\big)\\&=
\Tr\big(2 \g(f_t^\bot, g\i S  g\i).S\big) + \Tr\big(g\i.\nabla^2 f_t^{\bot}\big)\\&\qquad
+\Tr\big(g\i R^{\g}(f_t^\bot,Tf)Tf\big)+\L_{f_t^{\top}}\Tr^g(S)
\\&=
 \Tr\big(2\g(f_t^{\bot},g\i.S.g\i).S\big)-\Delta(f_t^{\bot})
\\&\qquad
 +\Tr\big(g\i R^{\g}(f_t^\bot,Tf)Tf\big)+\L_{f_t^{\top}}\Tr^g(S)\qedhere
\end{align*}
\end{proof}

\bibliographystyle{plain}

\end{document}